\documentclass[a4paper,twoside,12pt]{amsart}

\usepackage{amsfonts,amssymb,amscd,amsmath,enumerate}
\usepackage[greek, french, english]{babel}
\usepackage{graphicx}
\usepackage[all]{xy}
\usepackage{enumerate}
\usepackage{color}
\usepackage[T1]{fontenc}

\newtheorem{theorem}{Theorem}[section]
\newtheorem{proposition}[theorem]{Proposition}
\newtheorem{lemma}[theorem]{Lemma}

\theoremstyle{definition}
\newtheorem{definition}[theorem]{Definition}

\newtheorem{example}[theorem]{Example}
\theoremstyle{remark}
\newtheorem{remark}[theorem]{Remark}

\def\A{{\mathbf A}}

\def\Bb{{\mathcal B}}
\def\C{{\mathbf C}}
\def\Cc{{\mathcal C}}

\def\Ll{{\mathcal L}}

\def\N{{\mathbf N}}

\def\Pp{{\mathcal P}}

\def\Kk{{\mathcal K}}
\def\Q{{\mathbf Q}}

\def\R{\mathbf R}

\def\Z{{\mathbf Z}}
\def\Z{\mathbf Z}
\def\w{\mathbf w}

\def\int{\mathrm{int}}

\providecommand{\bysame}{\leavevmode\hbox to3em{\hrulefill}\thinspace}

\makeindex

\begin{document}
\bibliographystyle{amsplain}
\setcounter{tocdepth}{2}
\title{ON KAPLANSKY'S EMBEDDING THEOREM}
\author{Bernard Teissier}
\address{Universit\'e Paris Cit\'e and Sorbonne Universit\'e, CNRS, IMJ-PRG, F-75013 Paris, France}

\email{bernard.teissier@imj-prg.fr}


\pagestyle{myheadings}
\markboth{\rm Bernard Teissier}{\rm Kaplansky's embedding}
\renewcommand\rightmark{B. Teissier}
\renewcommand\leftmark{Kaplansky's embedding}

\subjclass[2020]{ 14E15,13A18}

\dedicatory{To the memory of a truly exemplary person and mathematician, Arkadiusz P\l oski}

\begin{abstract} Let $R$ be a complete equicharacteristic noetherian local domain with an algebraically closed residue field $k$. Let $\nu$ be a zero dimensional valuation centered in $R$ with value group $\Phi$. We prove that, if the value semigroup $\Gamma=\nu(R\setminus \{0\})$ of the valuation is finitely generated or the valuation is of rank one, the valuation $\nu$ is the restriction, via an embedding $R\subset k[[t^{\Phi_{\geq 0}}]]$, of the $t$-adic valuation $\nu_t$ of the ring of Hahn series with coefficients in $k$ and exponents in $\Phi_{\geq 0}$. This embedding is given by series $$\xi_i\mapsto\xi_i(t)=\rho_i t^{\gamma_i}+\sum_{\delta>\gamma_i}c^{(i)}_\delta t^\delta \ {\rm with}\  c^{(i)}_\delta\in k,\ \rho_i\in k^*,$$
where the $(\gamma_i)_{i\in I}$ constitute a minimal system of generators, indexed by an ordinal  $I\leq \omega$, of the semigroup $\Gamma=\break\nu(R\setminus \{0\})$. The $\xi_i\in R$ are well chosen representatives of a minimal system of homogeneous generators $(\overline\xi_i)_{i\in I}$ of the graded $k$-algebra ${\rm gr}_\nu R$ associated to $\nu$. The $\rho_i$ parametrize an isomorphism of graded $k$-algebras between ${\rm gr}_\nu R$ and $k[t^\Gamma]$.
 \end{abstract}
\maketitle
\noindent
\begin{section}{Introduction}

Among the problems studied by the founders of general valuation theory, in particular W. Krull and F.K. Schmidt, is the question of determining the extent to which a valued field is determined by its value group and residue field (see \cite{Roq}). This led to the notion of a maximal valued field, which has no valued extension with the same value group and residue field, called \textit{immediate}.   Krull raised the questions of whether an immediate embedding of a valued field in a maximal valued field is unique and whether in the equicharacteristic case a maximal valued field is a field of generalized power series. In the two papers \cite{Ka} and \cite{Ka2}, Kaplansky provided precise anwers to these questions by describing valued extensions using pseudo-convergent sequences in the sense of Ostrowski. The uniqueness holds provided some conditions called condition $(A)$ on the residue field and value group are satisfied (see \cite[Theorem 5]{Ka}) and in the equicharacteristic case the maximal field is a generalized power series field provided the same conditions hold (see \cite[Theorem 6]{Ka}) and in addition the value group is discrete (see \cite[Theorem]{Ka2}). Condition $(A)$ is vacuous if the residue field is of characteristic zero. The combination of these results is Kaplansky's embedding theorem. \par
More recently, the question of embedding \textit{valued noetherian local domains} in rings of generalized power series has become of interest, in particular in connexion with the problem of local uniformization of valuations in positive characteristic. Here a key role is played by the semigroup of values which the valuation takes on the local domain. The motivating example is as follows:\par Let $k$ be an algebraically closed field. A plane algebroid branch is given parametrically by two power series $x(t),y(t)\in k[[t]]$ without constant term whose exponents are coprime, which satisfy a relation $f(x(t),y(t))=0$ with $f(x,y)\in k[[x,y]]$ irreducible and without constant term.\par In characteristic zero, assuming that $f(0,y)\neq 0$, Newton's method produces a series $y(x)$ in rational power of $x$ with bounded denominators and  if $k=\C$ and the series $x(t),y(t)$ converge, one can deduce, as Puiseux was the first to do, a lot of geometric information, up to the bilipschitz geometry and the embedded resolution of singularities algorithm of the corresponding complex analytic branch, from the arithmetical properties of the exponents of $y(x)$. These properties are condensed in the \textit{Puiseux characteristic exponents}. \par\noindent From these properties, one can also deduce (see \cite[Chap. 1]{Z}) a minimal system of generators for the numerical semigroup $\Gamma\subset\N$ of the values taken by the $t$-adic valuation of $k[[t]]$ on the subalgebra $k[[x(t),y(t)]]\subset k[[t]]$. From this set of generators, one can conversely calculate the Puiseux exponents of the Newton series.\par
If $k$ is of positive characteristic, Newton's method does not work well at all; an expression of $y$ as a series in $x$ may have a very complicated set of exponents (see \cite{Ke}), with unbounded denominators, as is shown by the Artin-Schreier example which is recalled below. However, the semigroup $\Gamma$ still determines the embedded resolution of singularities process of the plane branch (see \cite[Remark 5.4]{Te2}), and therefore a lot of its geometry. We recommend \cite{C}, \cite{GB-P} and \cite{P-P} as very informative studies of plane algebroid curves in any characteristic.\par
If we accept to describe our curve by a parametrization in a higher dimensional space $\A^{b}(k)$, as $u_i=\xi_i(t),\ i=1,\ldots, b$, with $\xi_i\in k[[t]]$, we can do this in such a way that $\xi_i(t)=\rho_it^{\gamma_i}+\sum_{\delta>\gamma_i}c_\delta^{(i)}t^\delta$ with $\rho_i\in k^*,\ c_k^{(i)}\in k$, where $\gamma_1,\ldots ,\gamma_{b}$ is a minimal system of generators of the semigroup $\Gamma$.\par\noindent So we do not need to do any arithmetic with the exponents of the series $\xi_i(t)$ to extract the semigroup, which determines the equisingularity class of the plane branch in any characteristic (see \cite[Chapter 3]{C}).\par\noindent To recognize that this parametric representation is that of a plane branch, however, we now need to do arithmetic on the integers $\gamma_i$  and then work some more (see \cite[Appendix]{Z}); we have shifted the place where we have to work in order to describe the geometry of our curve.\par
The geometric idea here is that after eventual re-embedding of our branch in an affine space $\A^b(k)$, it is parametrized by a constant semigroup deformation of the parametrization of the monomial curve $u_i=t^{\gamma_i}$ which is the geometric realization of the semigroup $\Gamma$.\par\noindent
Thus re-embedded, the branch admits (see \cite{G-T}) an embedded resolution of singularities by a single toric modification of the ambient space $\A^b(k)$ which resolves the monomial curve. This resolution is blind to the characteristic (see \cite[\S 6.3]{Te1}).\par 
In higher dimension, a natural generalization of the parametric description of an algebroid branch is local uniformization of rational valuations of excellent local domains.\par We recall that a rational valuation on a local domain $(R,m)$  is the datum of an inclusion $R\subset R_\nu$ of $R$ in a valuation ring of its field of fractions, such that $R_\nu$ dominates $R$ in the sense that $m_\nu\cap R=m$ and the inclusion $k=R/m\subset R_\nu/m_\nu=k_\nu$ is an isomorphism. It corresponds to a $k$-rational point on the Zariski-Riemann manifold of $R$. If $k$ is algebraically closed, rational is the same as zero dimensional.\par A local uniformization is a regular local ring $R_1$ essentially of finite type over $R$ and dominated by $R_\nu$. For branches, $R_1=R_\nu$ but in general valuation rings are not noetherian.\par
What we show in this text generalizes the description of the ring of functions of the branch as a subring of $k[[t]]$ generated by the $\xi_i(t)$: it generalizes the parametric presentation of the branch in a suitable affine space exposing the semigroup generators as exponents of initial terms.\par For this reason we call it the \textit{embedded} Kaplansky embedding theorem.\par\noindent
Geometrically, it is a parametrization of the formal space corresponding to $R$ obtained as a constant semigroup deformation of the parametrization $u_i\mapsto\rho_it^{\gamma_i}$ of the generalized toric variety corresponding to ${\rm gr}_\nu R$. \par It is worth noting that for branches, the deformations which are obtained by deforming the parametrization (as opposed to deforming the equations) are exactly those which admit a simultaneous resolution of singularities (by normalization). See \cite[\S 3]{Te5}.\par We do not address the problem of the uniqueness of the parametric representation in the same way as Kaplansky. Suffice it to say that even for plane algebroid branches in characteristic zero, different parametric representations generating the same semigroup of $t$-adic values  may correspond to isomorphic one dimensional complete local domains with fields of fractions isomorphic to $k((t))$. See \cite[Chapter III]{Z}.\par The main difficulties arise from the fact that the semigroup $\Gamma=\nu(R\setminus\{0\})$ is not finitely generated in general. See section \ref{final} for the basic example.\par 
As a consequence, the "dimension" of the affine space where we re-embed our formal space may be a countable ordinal, which is $\leq \omega$ in the rank one case by \cite[Appendix 2]{Z-S}),
indexing a minimal system of generators of the semigroup $\Gamma$.\par\smallskip
We first recall how one proves, as in \cite[\S 7]{Te2}, the embedded Kaplansky embedding Theorem for rational valuations with finitely generated semigroup, whose value group is $\Z^{{\rm dim}R}$, on equicharacteristic noetherian local domains with an algebraically closed residue field. It is a consequence of local uniformization. Then we extend to all rational valuations of rank one with the additional assumption of completeness of the local domain by using their approximations by Abhyankar semivaluations with finitely generated semigroups.\par\noindent We recall that a semivaluation on $R$ is a valuation on a quotient of $R$ by a prime ideal. A valuation is a semivaluation such that only $0\in R$ gives a value $\infty$ larger than any element of $\Phi$. 
  \end{section}
\begin{section}{Hahn series}
\begin{definition}{\rm (Hahn in \cite{Hahn})}
Let $\Delta$ be a totally ordered abelian semigroup containing a smallest element $0$ and $A$ a commutative ring. The set of formal sums $s=\sum_{\delta\in E}a_\delta t^\delta$ where the set $E$ of exponents of the series is a well ordered subset of $\Delta$ and $a_\delta\in A$ for all $\delta\in E$ can be endowed with a commutative product by the usual multiplication rule of series. It is an $A$-algebra called the Hahn ring of series with exponents in $\Delta$ and coefficients in $A$ and denoted by $A[[t^\Delta]]$.\footnote{See \cite{Hahn}. Some call it the Mal'cev-Neumann ring because Hahn's definition was later generalized to the non abelian setting by these two authors.}
\end{definition}
In this text we will be interested in the case where $A$ is the residue field $k$ of a local subring $R$ of a valuation ring $R_\nu$ and $\Delta$ is the subsemigroup $\nu (R\setminus\{0\})$ of the non negative part $\Phi_{\geq 0}$ of a totally ordered abelian group $\Phi$ of finite rational rank $r$ which is the value group of the valuation.\par
If $R$ is a local noetherian subring dominated by $R_\nu$ and with residue field $k$, the semigroup $\Gamma =\nu(R\setminus \{0\})$ is well ordered and contained in the non negative part $\Phi_{\geq 0}$ of the value group $\Phi$ of $\nu$. We have an inclusion $$k[[t^\Gamma]]\subset k[[t^{\Phi_{\geq 0}}]],$$
where now $k[[t^\Gamma]]$ is the ring of all series with exponents in $\Gamma$.
\par\smallskip
The Hahn ring $k[[t^\Delta]]$ is endowed with a valuation defined as $\nu_\Delta(s)={\rm min}(\delta\vert a_\delta\neq 0)$. Its residue field is $k$ and its value group is the totally ordered abelian group $\Delta+(-\Delta)\subset \Phi$. It is spherically complete with respect to this valuation in the following sense (see \cite{O}, \cite{Ka}, \cite {Roq}):\par\medskip\noindent
A \emph{pseudo-convergent} sequence\footnote{They are also known as pseudo-Cauchy sequences. This concept is due to Ostrowski; see \cite[\S 11, p.368]{O}, \cite{Roq}.} of elements of a valued ring $(R,\nu)$, for example $(k_\nu[[t^\Delta]],\nu_\Delta)$, is a sequence $(y_\tau)_{\tau\in T}$ indexed by a well ordered set $T$ without last element, which satisfies the condition that whenever $\tau<\tau' <\tau"$ we have $\nu (y_{\tau'}-y_\tau)<\nu (y_{\tau"}-y_{\tau'})$.  One observes that if $(y_\tau)$ is pseudo-convergent, for each $\tau\in T$ the value $\nu (y_{\tau'}-y_\tau)$ is independent of $\tau'>\tau$ and can be denoted by $w_\tau$. An element $z$ is said to be a \emph{(pseudo-)limit} of this pseudo-convergent sequence if $\nu (y_{\tau'}-y_\tau)\leq \nu (z-y_\tau)$ for $\tau ,\tau'\in T,\ \tau<\tau'$. This is equivalent to $\nu (z-y_\tau)\geq w_\tau$ for all $\tau$.\par The balls $B(y_\tau, w_\tau)=\{y\in R\ \vert \nu (y-y_\tau)\geq w_\tau\}$ then form a strictly nested family of balls indexed by $T$ and their intersection is the set of limits of the sequence.\par
If $z$ is a limit of the sequence $(y_\tau)_{\tau\in T}$, Ostrowski defined in \cite[Chap. 11, p. 368]{O} (see \cite[p.304]{Ka})  the \textit{breadth} of the pseudo-convergent sequence: it is the ideal $\Bb\subset R$ of elements $z'$ such that $\nu (z')>w_\tau$ for all $\tau\in T$. Intuitively, the breadth corresponds to the "radius" of the intersection of the balls $B(y_\tau, w_\tau)$ with respect to the "ultrametric" $\nu_\Delta(y-x)$.
\begin{definition}\label{spher}{\rm (Ostrowski)}
A valued ring is \textit{spherically complete} if every pseudo-convergent sequence has a limit. Equivalently, every nested sequence of non-empty balls has a non-empty intersection.
\end{definition}
\begin{theorem}\label{complete}{\rm (See \cite[\S 1, \S 2]{Ka}, \cite{Ri})} The valued ring $k[[t^\Delta]]$ is spherically complete.\hfill$\square$
\end{theorem}

\begin{example} 1) Assume that our value group $\Phi$  has rank $h>1$ and let 
$$(0)=\Psi_h\subset \Psi_{h-1}\subset \cdots \subset \Psi_1\subset \Psi_0=\Phi$$ be the nested set of its convex subgroups. Let $(\delta_i)_{i\in T}$ be an infinite well ordered set of elements of $\Phi_{\geq 0}$ without a last element contained in a strict convex subgroup $\Psi_j,\ j\geq 1$. Then the sum $y=\sum_{i\in T}t^{\delta_i}$ is a limit of the pseudo-convergent sequence $(y_\tau=\sum_{i\leq\tau} t^{\delta_i})_{\tau\in T}$ of elements of $k[[t^{\Phi_{\geq 0}}]]$, but if we now take any $t^{\delta}$ such that $\delta\notin \Psi_j$, then $y+z$ is another limit and if we consider the sequence of finite sums $\tilde y_\tau=\sum_{i\leq\tau} (t^{\delta_i} +t^\delta)$, it is still pseudo-convergent and certainly cannot have a (usual) limit in our ring, but it still has $\sum_{i\in T}t^{\delta_i} $ as a limit.\par\noindent
2) The original, and more significant, example is that of the truncations of the roots of an Artin-Schreier equation attached to a prime number $p$, such as $$y^p-x^{p-1}(1+y)=0,$$
in $K[y]$, with $K$ the perfect closure $\bigcup_{n\geq 0}k(x^{\frac{1}{p^n}})$ of $k(x)$, where $k$ is an algebraically closed field of characteristic $p$, which are the  conjugates $\zeta +ax,\ 1\leq a\leq p-1,$ of $$\zeta=x^{1-\frac{1}{p}}+x^{1-\frac{1}{p^2}}+\cdots +x^{1-\frac{1}{p^i}}+\cdots\in k[[x^{(\frac{1}{p^\infty}\Z)_{\geq 0}}]].$$
The exponents form a well ordered set of rational numbers and the truncations form a pseudo-convergent sequence for the $x$-adic valuation of $k[[x^{(\frac{1}{p^\infty}\Z)_{\geq 0}}]]$, which obviously does not converge in the usual sense.\par
In this case the breadth of the pseudo-convergent sequence is the ideal $\{s\in k[[x^{(\frac{1}{p^\infty}\Z)_{\geq 0}}]]\vert \nu_x(s)\geq 1\}$.\par\noindent
This is an example of the method invented by Ostrowski and pursued by Kaplansky: the sequence of truncations is a pseudo convergent sequence of algebraic type of elements of $K$ equipped with the $x$-adic valuation, which does not have a limit in $K$. The $x$-adic valuation of $k(x)$ has a unique extension $\nu_x$ to $K$ and the extension of $\nu_x$ from $K$ to $ K(\zeta)$ described in \cite[Theorem 3]{Ka} is immediate: it changes neither the value group $\frac{1}{p^\infty}\Z$ nor the residue field $k$.
Other interesting examples may be found in \cite[Chap. III]{O}, \cite{R} and \cite[\S 4]{S2}, as well as \cite{C-M-T}.
\begin{lemma}\label{Kap}{\rm (Ostrowski)} Given a pseudo-convergent sequence $(y_\tau)_{\tau\in T}$ with breadth $\Bb$, and a limit $z$ of the sequence, all other limits are of the form $z+z'$ with $z'\in \Bb$.
\end{lemma}
\begin{proof} See \cite[No. 72, p. 386]{O}, \cite[Lemma 3]{Ka}.
\end{proof}
We note that the exponents of $\zeta$ have an accumulation point in $\R$. The series $\zeta$ belongs to the integral closure $\widetilde{k[[x]]}$ of $k[[x]]$ in $K$. The semigroup of values on $\widetilde{k[[x]]}$ of $\nu_x$ also has accumulation points in $\R$ (see \cite[\S 10]{Te2}) but this does not contradict the results of \cite[Section 3]{C-T} about the absence of accumulation points in semigroups of values of valuations on noetherian local domains because $\widetilde{k[[x]]}$ is not noetherian. \end{example}
\end{section}
\begin{section}{The valuative Cohen Theorem and its applications}\label{secCoh}
 In this section we do not need the residue field to be algebraically closed.\par
Since $R$ is noetherian, the semigroup $\Gamma =\nu(R\setminus \{0\})$ is well ordered and admits a natural minimal system of generators
$$\Gamma=\langle \gamma_1,\ldots ,\gamma_i,\gamma_{i+1}, \ldots\rangle,$$
where $\gamma_{i+1}$ is the smallest element of $\Gamma$ which is not in the semigroup generated by $\gamma_1,\ldots ,\gamma_i$. It is shown in 
\cite[Appendix 3, Proposition 2]{Z-S} that this set of generators is indexed by an ordinal $I \leq \omega^h$ where $h$ is the rank (or height) of the valuation, which is $\leq {\rm dim} R$.\par
Let $(u_i)_{i\in I}$ be a set of variables indexed by the generators $(\gamma_i)_{i\in I}$ of $\Gamma$. It is shown in \cite[\S 4]{Te2} and more precisely in \cite{Te3} that the set of formal sums $\sum d_eu^e$ with $d_e\in k$ and $u^e$ a monomial in the variables $u_i$ is, with the usual multiplication rule for series, a generalized power series $k$-algebra with all the properties of power series rings, except of course noetherianity if $\Gamma$ is not finitely generated. It is regular in any reasonable sense. It is endowed with a weight by giving $u_i$ the weight $\gamma_i$. We denote it by $\widehat{k[(u_i)_{i\in I}]}$. On this local domain the weight determines a valuation $w$; the value of a series is the weight of its terms of lowest weight. This is, by construction, a rational \textit{monomial} valuation. We shall sometimes write $w(m)$ for $w(u^m)$.\par It is proved in \cite[\S 4]{Te2} that $\widehat{k[(u_i)_{i\in I}]}$ is spherically complete with respect to its valuation $w$. The morphism $$\widehat{k[(u_i)_{i\in I}]}\to k[[t^{\Phi_{\geq 0}}]],\ u_i\mapsto t^{\gamma_i}$$ 
is a continuous morphism of spherically complete valued $k$-algebras, whose image is $k[[t^\Gamma]]$.\par\noindent Its kernel is the closure for the $w$-adic topology of the ideal generated by binomials $u^m-u^n$ corresponding to a system of generators of the relations $\sum m_i\gamma_i=\sum n_j\gamma_j$ between the generators $\gamma_i$.\par\noindent
 The continuous surjective morphism $\widehat{k[(u_i)_{i\in I}]}\to k[[t^\Gamma]]$ may be seen as encoding the minimal embedding of the formal space corresponding to $k[[t^\Gamma]]$ in a non singular space.\par
We now recall, from \cite[\S 4]{Te2}, the statement of the valuative Cohen Theorem, which will allow us to approximate our rational valuation by semivaluations with finitely generated semigroup. \par
To a valued local domain $R\subseteq R_\nu$ with valuation ring $R_\nu$, value group $\Phi$ and semigroup $\Gamma\subset \Phi_{\geq 0}$ is associated a graded ring as follows. Define as in \cite[Appendix]{Z} and  \cite[Section 2.1]{Te1} the filtration of $R$ associated to the valuation $\nu$:
$$\Pp_\varphi (R)=\{x\in R/\nu(x)\geq \varphi\} \ {\rm and}\ \Pp_\varphi^+(R)=\{x\in R/\nu(x)>\varphi\},$$
and the associated graded ring  $${\rm gr}_\nu R=\sum_{\varphi\in\Phi}\Pp_\varphi(R)/\Pp_\varphi^+(R)\subseteq \sum_{\varphi\in\Phi}\Pp_\varphi(R_\nu)/\Pp_\varphi^+(R_\nu)={\rm gr}_\nu R_\nu,$$
where the non-zero homogeneous components of ${\rm gr}_\nu R$ must have degree $\varphi\in\Gamma$. We note that $\Pp_0(R)=R$, and $\Pp_0^+(R)=m$ if $m_\nu\cap R=m$.\par
As shown in \cite[\S 4]{Te1}, by the defining property of valuation rings, each homogeneous component of ${\rm gr}_\nu R_\nu$ is a one-dimensional vector space over $k_\nu=R_\nu/m_\nu$. If the valuation is rational, the non-zero homogeneous components of ${\rm gr}_\nu R$ are one-dimensional vector spaces over $k=R/m=k_\nu$ and therefore ${\rm gr}_\nu R$ is a quotient of a polynomial ring $k[(U_i)_{i\in I}]$ in such a way that if we fix a minimal system of homogeneous generators $(\overline\xi_i)_{i\in I}$ of the $k$-algebra  ${\rm gr}_\nu R$, which are in bijection with the generators $(\gamma_i)_{i\in I}$ of $\Gamma$,  the kernel of the surjective morphism
$$k[(U_i)_{i\in I}]\longrightarrow {\rm gr}_\nu R,\ U_i\mapsto\overline\xi_i$$
is generated by binomials $(U^{m^\ell}-\lambda_\ell U^{n^\ell})_{\ell \in L}$. Indeed, ${\rm gr}_\nu R$ is isomorphic as a graded $k$-algebra to the semigroup algebra $k[t^\Gamma]$ of the semigroup $\Gamma$ with coefficients in $k$, which corresponds to the case where all $\lambda_\ell$ are equal to $1$ (see \cite[Section 3.1]{Te3}). The generators $\overline \xi_i$ have degree $\gamma_i$ and each is determined modulo multiplication by an element of $k^*$.\par
The generators $\overline \xi_i$ being fixed, an isomorphism of graded $k$-algebras between ${\rm gr}_\nu R$ and $k[t^\Gamma]$ is necessarily of the form $$\overline\xi_i\mapsto \rho_it^{\gamma_i},$$ 
where the $\rho_i$ are non-zero solutions $u_i=\rho_i$ of the system of binomial equations  $(u^{m^\ell}-\lambda_\ell u^{n^\ell}=0)_{\ell \in L}$. Since $k$ is algebraically closed, one can show by transfinite induction that such solutions exist. If $\Gamma$ is finitely generated, say with $b$ generators, the $\rho_i$ are the coordinates of a point of a $(k^*)^r$ torus orbit in $\A^b(k)$, where $r$ is the rational rank of $\Phi$, which is also the Krull dimension of ${\rm gr}_\nu R$ (see \cite[Proposition 3.1]{Te1}).

\begin{theorem}{\rm (The valuative Cohen Theorem)}\label{Cohen}\par\noindent
1) Let $R$ be a complete equicharacteristic noetherian local domain and let $\nu$ be a rational valuation of $R$. We fix a field of representatives $k\subset R$ of the residue field $R/m$ and a minimal system of homogeneous generators $(\overline\xi_i)_{i\in I}$ of the graded $k$-algebra ${\rm gr}_\nu R$. There exist choices of representatives $\xi_i\in R$ of the $\overline\xi_i$ such that the map $u_i\mapsto \xi_i$ determines a surjective morphism of $k$-algebras 
 $$\pi\colon \widehat{k[(u_i)_{i\in I}]}\longrightarrow R$$
 which is continuous with respect to the topologies associated to the filtrations by weight and by valuation respectively. The associated graded morphism with respect to these filtrations is the morphism
 $${\rm gr}_w\pi\colon k[(U_i)_{i\in I}]\longrightarrow {\rm gr}_\nu R,\ \ U_i\mapsto\overline\xi_i$$
whose kernel is a prime ideal generated by binomials $(U^{m^\ell}-\lambda_\ell U^{n^\ell})_{\ell\in L}$, with $\lambda_\ell\in k^*$.\par\noindent
If the semigroup $\Gamma=\nu(R\setminus\{0\})$ is finitely generated or if the valuation $\nu$ is of rank one, one may take any system of representatives $(\xi_i)_{i\in I}$. 
\par\smallskip\noindent
2) The kernel $F$ of $\pi$ is the closure of the ideal generated by elements $$F_\ell=u^{m^\ell}-\lambda_\ell u^{n^\ell}+\sum_{w(m)>w(m^\ell)}c_mu^m,\ c_m\in  k$$ as $U^{m^\ell}-\lambda_\ell U^{n^\ell}$ runs through a set $L$ of generators of the kernel of ${\rm gr}_w\pi$. In a slightly generalized sense, these generators form a standard basis of the kernel of $\pi$, with respect to the weight $w$.\par\smallskip\noindent
3) One can choose a minimal system of binomial generators for the kernel of ${\rm gr}_w\pi$ and the $F_\ell$ may be chosen so that each of them involves only finitely many of the $u_i$.\par\smallskip\noindent4) Denoting by $J\subset I$ a set of variables $(u_i)_{i\in J}$ minimally generating the maximal ideal of $R$, the kernel of $\pi$ is the closure of the ideal generated by some of the series $(F_\ell)_{\ell\in L}$ which we write as $$F_i=u^{m^i}-\lambda_iu^{n^i}+\sum c_m^{(i)}u^m-u_i,\ {\rm for}\ i\in I\setminus J,\ {\rm where}\ w(m^i)<w(m)<\gamma_i,$$
and finitely many series
 $$F_q=u^{m^q}-\lambda_qu^{n^q}+\sum c_m^{(q)}u^m, \ {\rm with}\ w(m)>w(m^q),\ q=1,\ldots ,f.$$
 
\end{theorem}
\noindent This is proved in \cite[\S 4]{Te2} and \cite{Te3}. For rank one valuations the statement is a consequence of Chevalley's Theorem on filtrations of complete noetherian local rings (see \cite[Chap. IV, \S 2, No. 5, Cor. 4]{B}) and the general case is obtained by induction on the rank. \par 
Intuitively, the $F_i$ are necessary equations because in the ring $R$ the $(\xi_i)_{i\in I\setminus J}$ are series in the $(\xi_i)_{i\in  J}$, and the $F_q$ become, after elimination of the variables $(\xi_i)_{i\in I\setminus J}$ through the equations $F_i$, the equations defining the ring $R$ in a formal power series ring $k[[(u_i)_{i\in  J}]]$ in accordance with the usual Cohen Theorem (see \cite{Co}). The series $(F_\ell)_{\ell\in L}$ then constitute a (generalized) standard basis for the ideal topologically generated by the $F_i,F_q$. \par\medskip 
\end{section}
\begin{section}{\rm Torific embeddings for rational valuations of rank one with finitely generated semigroup}
In this section we assume that $\nu$ is a rational valuation of rank one on the complete equicharacteristic noetherian local domain $R$, that the residue field $k=k_\nu$ is algebraically closed and that the semigroup $\Gamma$ is finitely generated, say $\Gamma=\langle\gamma_1,\ldots ,\gamma_b\rangle$. By \cite[beginning of \S 7]{Te2}, this implies that the valuation $\nu$ is Abhyankar (in our case this means that ${\rm dim}{\rm gr}_\nu R={\rm dim}R$) and in particular its value group is $\Z^{{\rm dim}R}$. We set $r={\rm dim}R$ and $B =\{1,\ldots ,b\}$. The kernel of the map $k[(U_i)_{i\in B}]\to {\rm gr}_\nu R$ is finitely generated so the set $L$ is finite.\par\noindent We fix an ordered embedding of the value group $\Phi$ in $\R$. \par
According to Theorem \ref{Cohen} there are elements $(\xi_i)_{i\in B}$ in the maximal ideal of $R$ whose initial forms $(\overline\xi_i)_{i\in B}$ constitute a minimal set of generators of the $k$-algebra ${\rm gr}_\nu R$ and such that the morphism $$  \pi\colon k[[(u_i)_{i\in B}]]\longrightarrow R,\ u_i\mapsto\xi_i\leqno{(*)}$$ is a continuous surjective morphism of complete noetherian $k$-algebras whose kernel is generated by series $F_\ell$ as in $2)$ of Theorem \ref{Cohen}. \par\smallskip
The embedding of the algebroid space corresponding to $R$ into $\A^b(k)$ corresponding to this morphism is called a \textit{torific embedding}.\par\smallskip
Given a surjective morphism of rings $g\colon S\to R$ and a filtration $\Pp=(\Pp_\varphi)_{\varphi\in \Delta}$ of $S$ by ideals, the image filtration in $R$ is the filtration by the ideals $g(\Pp_\varphi)$. The order of an element of $R$ with respect to $g(\Pp)$ is the maximum of the orders with respect to $\Pp$ of its preimages in $S$. This maximum exists in the framework we are in. See \cite[Corollary 3.4, Proof of a)]{Te2}.\par  The image by the morphism $(*)$ of the filtration of $k[[(u_i)_{i\in B}]]$ by weight is the $\nu$-filtration of $R$. Since the filtration by weight can also be considered as associated to the monomial valuation $w$ by smallest weight on $k[[(u_i)_{i\in B}]]$, we may say that \textit{The valuation $\nu$ is the image by $\pi$ of the monomial weight valuation of $k[[(u_i)_{i\in B}]]$}.
\par\smallskip\noindent
\textbf{Recall} that :\par\noindent
It is proved in \cite[Theorem 11]{KKMS}, \cite[Chap. V, \S 6]{Ewald} and \cite[Chapter 11]{C-L-S} that we have:
\begin{proposition}\label{cones}  Given a finite collection of rational convex cones in $\check\R^b$ contained in a rational convex cone $C$, there are regular fans $\Sigma$ with support $C$ which are compatible with all the cones of the collection in the sense that the intersection of any cone $\sigma$ of $\Sigma$ with a cone of the collection is a face of $\sigma$.
\end{proposition}\noindent
 It is shown in \cite[Proposition 3.3, b)]{Te2} that we have:
\begin{proposition}\label{Resol} For a fixed torific embedding and a choice of generators $(F_\ell)_{\ell\in L}$ of the kernel of the morphism $(*)$ there exist finitely many rational strictly convex cones $(C_\ell)_{\ell\in L}$ in the space $\check \R_{\geq 0}^b$ such that any regular fan $\Sigma$ with support $\check\R_{\geq 0}^b$ which is compatible with those cones and with the traces in $\check \R_{\geq 0}^b$ of the hyperplanes $(H_\ell)_{\ell\in L}$ dual to the vectors $(m^\ell-n^\ell)_{\ell\in L}$ has the following properties:\par\noindent
\begin{enumerate}
\item The fan $\Sigma$ contains a unique $r$-dimensional rational convex cone $\sigma_\w=\langle a^1,\ldots ,a^r\rangle$ generated by primitive integral vectors $a^1,\ldots ,a^r$ and containing  the {\rm weight vector} $\w=(\gamma_1,\ldots ,\gamma_b)\in \check\R^b$ in its relative interior.
\item In any chart $Z(\sigma)\subset Z(\Sigma)$ corresponding to a cone $\sigma=\langle a^1,\ldots ,a^b\rangle$ of maximal dimension containing $\sigma_\w$ the vectors $a^1,\ldots ,a^r$ generating the convex cone $\sigma_\w$ correspond naturally to coordinates $y_1,\ldots ,y_r$.
\item The birational toric morphism $\varpi(\Sigma)\colon Z(\Sigma)\to \A^b(k)$ is such that the strict transform $\tilde\Cc$ of the toric subvariety $\Cc$ of $\A^b(k)$ defined by the binomials $(u^{m^\ell}-\lambda_\ell u^{n^\ell})_{\ell\in L}$ is non singular.
\item In a chart $Z(\sigma)$ as above the center of the valuation $w$ on the birational model $\tilde\Cc$ of $\Cc$ is the point defined by $$y_1=\ldots =y_r=0, y_{r+1}=c_{r+1},\ldots ,y_b=c_b\ \ {\rm with}\ c_j\in k^*$$ and $y_1,\ldots ,y_r$ are local coordinates on $\tilde\Cc$ with rationally independent values.
\item The coordinates $y_1,\ldots ,y_b$ are Laurent monomials in the coordinates $u_i$, say $y_j=u^{\alpha_j}$ and the ring $R_1=R[\xi^{\alpha_1},\ldots, \xi^{\alpha_b}]$ is a subring of the valuation ring $R_\nu$. The local ring $\tilde R=(R_1)_{m_\nu\cap R_1}$ is called the algebraic transform of $R$ by the toric modification $\varpi(\Sigma)$.
\item  The local ring $\tilde R$ is regular and its maximal ideal is generated by the images of $y_1,\ldots ,y_r$, namely the images of $\xi^{\alpha_1},\ldots, \xi^{\alpha_r}$ in the localization of $R_1$ by $m_\nu\cap R$.
\item The values of $\xi^{\alpha_1},\ldots, \xi^{\alpha_r}$ are rationally independent so that the valuation $\nu$ on $\tilde R$ is induced by the monomial valuation on its completion $k[[y_1,\ldots ,y_r]]$ determined by the values of the $y_i$, whose semigroup is $\N^r$. These two valuations induce the valuation $\nu$ on $R$.
\end{enumerate}
\end{proposition}
\begin{definition}\label{mizer}The datum of a torific embedding and a regular fan $\Sigma$ with support $\check\R^b_{\geq 0}$ satisfying the conditions of Proposition \ref{Resol}
 above will be called a \textit{torific uniformizer} for $\nu$.
 \end{definition} 
 It follows from what we have just recalled that torific uniformizers exist for any rational valuation with finitely generated semigroup on a complete equicharacteristic noetherian local domain with an algebraically closed residue field.

 Let  $\Sigma$ be a regular fan with support the first quadrant $\check \R^b_{\geq 0}$ and compatible with the hyperplanes $(H_\ell)_{\ell\in L}$ dual to the vectors $m^\ell-n^\ell$ and the cones $(C_\ell )_{\ell\in L}$, and let $\sigma_\w$ be an $r$-dimensional cone of $\Sigma$ containing the weight vector $\w=(w(u_1),\dots ,w(u_b))$.
 \begin{lemma}\label{weighin}The cone $\sigma_\w$ of $\Sigma$ is uniquely determined and contained in the \textit{weight cone} $W=(\bigcap_{\ell\in L}H_\ell) \bigcap\check \R^b_{\geq 0}$. 
 \end{lemma}
 \begin{proof}Since the vectors $m^\ell-n^\ell$ have entries of different signs, their linear duals $H_\ell$ meet the interior of the first quadrant of $\check\R^b_{\geq 0}$ so that $W$ is of dimension $r$. Since the coordinates of $\w$ generate a group of rational rank $r$, by \cite[Lemma 3.10]{Te2} the vector $\w$ cannot be contained in a rational cone of dimension $<r$. The intersection of two cones of $\Sigma$ containing $\w$ or, if $\sigma_\w$ was not contained in $W$, its intersection with $W$, would be such a cone.
 \end{proof}
Let $\sigma\in\Sigma$ be a cone of maximal dimension containing $\sigma_\w$. Write $\sigma=\langle a^1,\ldots ,a^r,a^{r+1},\ldots ,a^b\rangle$. The birational toric morphism $Z(\sigma)\to \A^b(k)$ is described in coordinates as:
$$\begin{array}{lr}
\ \ \ \ \ u_1=y_1^{a^1_1}\ldots y_r^{a^r_1}y_{r+1}^{a^{r+1}_1}\ldots y_b^{a^b_1}\\
\ \ \ \ \ .\\
\ \ \ \ \ .\\
\ \ \ \ \ .\\
\ \ \ \ \ u_i=y_1^{a^1_i}\ldots y_r^{a^r_i}y_{r+1}^{a^{r+1}_i}\ldots y_b^{a^b_i}\\
\ \ \ \ \ .\\
\ \ \ \ \ .\\
\ \ \ \ \ .\\
\ \ \ \ \ u_b=y_1^{a^1_b}\ldots y_r^{a^r_b}y_{r+1}^{a^{r+1}_b}\ldots y_b^{a^b_b}\\

\end{array}$$
From this array follows that the valuations of the $y_i$ are the coordinates of the vector $\w\in\sigma$ with respect to $a^1,\ldots ,a^b$. Since $\w\in\sigma_\w$, the values of $y_{r+1},\ldots ,y_b$ are zero. Let $\xi_1,\ldots ,\xi_b$ be elements of $R$ lifting generators $\overline\xi_1,\ldots ,\overline\xi_b$ of the $k$-algebra ${\rm gr}_\nu R$ according to the valuative Cohen Theorem. 
Having fixed a torific embedding of $R$, i.e., a morphism $(*)$ with generators of its kernel, we have seen that such an embedded local uniformization provides us with a composed morphism of valued rings $$R\subset \tilde R\subset k[[y_1,\ldots ,y_r]]\simeq k[[t^{\N^r}]].$$ 
We know that the corresponding inclusion of associated graded rings has to correspond to an inclusion ${\rm gr}_\nu R\simeq k[t^\Gamma]\subset k[t^{\N^r}]$.

 \begin{proposition}\label{OstroAbh}{\rm (Embedded Kaplansky embedding for rational valuations with finitely generated semigroup)}\par\noindent
Let $R$ be a complete equicharacteristic noetherian domain and $\nu$ a rational valuation of $R$ with finitely generated semigroup. Fix a field of representatives $k\subset R$, a torific embedding as above and a torific uniformizer $\Sigma$ in $\check \R^b_{\geq 0}$.
\begin{enumerate}
\item  The ring $R$ is contained in a uniquely determined regular local ring $$\tilde R=R[\xi^{\alpha_1},\ldots, \xi^{\alpha_b}]_{m_\nu\cap R[\xi^{\alpha_1},\ldots, \xi^{\alpha_b}]}$$ as above, endowed with a rational monomial valuation which induces the valuation of $R$. 
\item The monomial valuation of $\tilde R$ extends uniquely to a rational monomial valuation of its completion $k[[ y_1,\ldots , y_r]]$ and we may view this last ring as $k[[t^{\N^r}]]$, which can itself be identified with a subalgebra of $k[[t^{\Phi_{\geq 0}}]]$, with $\Phi=\Z^r$, by sending $ y_i$ to $t^{\nu( y_i)}$. The valuation on $k[[ y_1,\ldots , y_r]]$ and therefore also the valuation $\nu$ on $R$, is induced by the $t$-adic valuation of $k[[t^{\Phi_{\geq 0}}]]$. The image in $k[[t^{\N^r}]]$ or $k[[t^{\Phi_{\geq 0}}]]$ of each $x\in R$ is of the form $$x(t)=\rho t^{\nu (x)}+\sum_{\delta>\nu (x)}c_\delta t^\delta\  {\rm with} \ c_\delta\in k,\ \rho\in k^*.$$
In particular, the image of $\xi_i$ is of the form $\rho_it^{\gamma_i}+\sum_{\delta>\gamma_i}c^{(i)}_\delta t^\delta$.\par\smallskip\noindent
\item The image in $k[[ y_1,\ldots , y_r]]$ of an element $x\in R$ is the strict transform of $x$ by the ambient toric morphism inducing the $\nu$-birational morphism corresponding to the inclusion $R\subset \tilde R$.
\item The center of the valuation $\nu$ on $\tilde R$ corresponds to the point $y_1=\cdots =y_r=0,y_{r+1}=c_{r+1},\ldots ,y_b=c_b$ with $c_j\in k^*$.\par\noindent The image in $k[[ y_1,\ldots , y_r]]$ of each $\xi_i\in R$ is of the form
$$\xi_i\mapsto y_1^{a^1_i}\cdots y_r^{a^r_i}(c_{r+1}+z_{r+1})^{a^{r+1}_i}\cdots (c_b+z_b)^{a^b_i},$$
where the $z_j$ for $r+1\leq j\leq b$ are the images in $\tilde R$ of the $y_j-c_j$ and are power series without constant term in the $y_1,\ldots , y_r$, and $\sum_{j=1}^ra^j_i\nu(y_j)=\gamma_i$.
\item The value of $\rho_i$ is $c_{r+1}^{a^{r+1}_i}\ldots c_b^{a^b_i}$.\qed
\end{enumerate}
\end{proposition}
\begin{subsection}{Examples}
\begin{example}\label{A_S}Let $k$ be an algebraically closed field of characteristic $p$. By \cite[Section 10]{Te2},\ the ring $k[[x]][y]/(y^p-x^{p-1}(1+y))$ valued by the unique extension of the $x$-adic valuation, embeds in $k[[t]]=k[[t^{\Z_{\geq 0}}]]$ by $x\mapsto \frac{t^p}{1-t^{p-1}},\ y\mapsto \frac{t^{p-1}}{1-t^{p-1}}$. If we want to represent $y$ as a series in $x$, we have to take $y=\zeta$, where $\zeta$ is the series we have seen above, which is not a Puiseux series and is only pseudo-convergent. 
\end{example}
In this paper we do not address the problem of uniqueness of the embedded Kaplansky embeddings which we build. The only step we take in this direction is the following:
\begin{example}\label{pi}{\rm \textit{This is an elaboration of an example kindly sent by referee $X_0$.}}\par
Let $k$ be an algebraically closed field. Let $R=k[[u_1,u_2]]$. Let $\ell,q$ be coprime integers not divisible by the characteristic of $k$ and $\nu$ a valuation on $R$ with $\nu (u_1)=\ell$, $\nu(u_2)=q$ and
$\nu(u_1^q-u_2^\ell)=q\ell+\pi$. Set $u_3=u_1^q-u_2^\ell$.\par\noindent
Let $a,b$ be positive integers such that $b\ell -aq=1$ and $t$ a positive integer.\par\noindent Consider in $\check \R^3$ the vectors $v_1=(\ell,q, q\ell+3),v_2=(a,b,t), v_3=(\ell,q, q\ell+4)$. The determinant ${\rm det}(v_1,v_2,v_3)$ is equal to one for any value of $t$. The convex cone $\sigma$ generated by $(v_1,v_2,v_3)$ in $\check \R^3$ is regular and contains the vector $(\ell,q, q\ell+\pi)$ in one of its faces.\par\medskip\noindent
The toric chart $\varpi (\sigma)\colon Z(\sigma)\to \A^3(k)$ defined by $\sigma$ is given in coordinates by:\par\noindent
$u_1=y_1^\ell y_2^ay_3^\ell,\  u_2=y_1^q y_2^by_3^q,\ u_3=y_1^{q\ell +3}y_2^t y_3^{q\ell +4}$,
and the transform of $u_1^q-u_2^\ell$ is $y_1^{q\ell}y_2^{aq}y_3^{q\ell}(y_2-1)$, while the transform of $u_1^q-u_2^\ell-u_3$ is
$$y_1^{q\ell}y_2^{aq}y_3^{q\ell}(y_2-1-y_1^3y_2^{t-aq}y_3^4).$$
We choose $t=aq$ and set $w=y_2-1$ so that the completion of the transform $\tilde R$ of our ring $R$ is $$k[[y_1,w,y_3]]/(w-y_1^3y_3^4).$$
From the expressions of $u_1,u_2,u_3$ we find that $\nu(y_1)+\nu(y_3)=1$ and that $\nu(y_1)=4-\pi, \nu(y_3)=\pi-3$ so that $\nu(w)=\pi$.\par
We are in the situation where the transform $\tilde R$ of $R$ is a regular local ring with coordinates $y_1,y_3$ and their valuations are rationally independent.\par
The morphism of $k$-algebras determined by $y_1\mapsto t^{4-\pi},y_3\mapsto t^{\pi-3},w\mapsto t^\pi$ is an embedding of $\tilde R$ into $k[[t^{\Phi_{\geq 0}}]]$ with $\Phi = \Z+\Z\pi$. It induces the injection $R\subset k[[t^{\Phi_{\geq 0}}]]$ determined by $u_1\mapsto (1+t^\pi )^at^\ell, u_2\mapsto (1+t^\pi )^bt^q$, so that $u_1^q-u_2^\ell\mapsto -(1+t^\pi)^{aq}t^{q\ell+\pi}=-u_1^qt^\pi$.\par
Let us now consider another embedding of $R$ into $k[[t^{\Phi_{\geq 0}}]]$, corresponding to a regular cone $\sigma'$ generated by the vectors $v_1,v'_2=(a',b',a'q), v_3$, with $b'\ell-a'q=1$.\par\noindent

To prove that the two embeddings differ by an inner valuation-preserving $k$-automorphism of $k[[t^{\Phi_{\geq 0}}]]$, it suffices (see \cite[\S 4.2]{K-S}) to prove the existence of a morphism of groups $u\colon \Phi\to U$, where $U$ is the multiplicative group of units of $k[[t^{\Phi_{\geq 0}}]]$, such that:\par\medskip\noindent
 $\bullet$ For any series $\sum_\varphi c_\varphi t^\varphi\in k[[t^{\Phi_{\geq 0}}]]$ the series $\sum_\varphi c_\varphi u(\varphi) t^\varphi $ belongs to $k[[t^{\Phi_{\geq 0}}]]$.\par

$$\leqno{\bullet}\ \ \ \ \ \ \ \ \ \ \ \ \ \ \ \ \ \ \ \ (1+u(\pi)t^\pi)^a u(\ell)t^\ell=(1+t^\pi)^{a'}t^\ell,$$ $$(1+u(\pi)t^\pi)^bu(q)t^q=(1+t^\pi)^{b'}t^q,\ {\rm and}$$
$$ (1+u(\pi) t^\pi)^{aq}u(q\ell+\pi)t^{q\ell+\pi}= (1+t^\pi)^{a'q}t^{q\ell+\pi}.$$
The automorphism will then be described by $\sum_\varphi c_\varphi t^\varphi \mapsto \sum_\varphi c_\varphi u(\varphi) t^\varphi $.\par
 We see that $u(\pi)=1$ and $u(\ell)=(1+t^\pi)^{a'-a}, u(q)=(1+t^\pi)^{b'-b}$ satisfy these conditions. We have that $(b'-b)\ell-(a'-a)q=0$, so that $\frac{b'-b}{q}=\frac{a'-a}{\ell}$ is an integer $n$.\par The morphism $u\colon\Z+\Z\pi \to U$ determined by  $u(1)=(1+t^\pi)^n, u(\pi)=1$ gives the desired automorphism provided that $$\sum_{r,s}c_{r,s}(1+t^\pi)^{nr}t^{r+s\pi}\in k[[t^{\Phi_{\geq 0}}]]$$ for any series $\sum_{r,s}c_{r,s}t^{r+s\pi}\in k[[t^{\Phi_{\geq 0}}]]$. But the set of exponents of this series differs from the original set by elements of the semigroup generated by $1$ and $\pi$ so that it is well ordered.
 \end{example}
   \end{subsection}
\begin{remark}(from \cite[section 4.3]{Te1}) There is an explicit algorithm describing how the semigroups of values of the $k[[y_1,\ldots ,y_r]]$ fill up the semigroup $\Phi_{\geq 0}$ as the cones $\sigma_\w$ become smaller and smaller. By point 1) of the recall we can view them as belonging to a sequence of refinements of fans with support $\R_{\geq 0}^b$.\par Let $\tau_1,\ldots ,\tau_r$ be rationally independent positive real numbers and let $\Phi\subset \R$ be the group which they generate, ordered by the order of $\R$. Let  $\sigma_0$ be a strictly convex regular rational cone of dimension $r$ in $\check\R_{\geq 0}^r$  containing the vector $\w=(\tau_1,\ldots ,\tau_r)$. The Jacobi-Perron algorithm produces a nested sequence $$\sigma_0\supset\ldots\supset \sigma_h\supset\sigma_{h+1}\supset \ldots\ldots\ni\w$$
of strictly convex regular rational cones of dimension $r$ converging to $ \w$.\par\noindent
The semigroup $\Phi_{\geq 0}$ is isomorphic as a semigroup to the semigroup of integral points in $\R^r$ of the half-space $\sum_{i=1}^r\tau_ia_i\geq 0$ which is the convex dual $\check \w$ of the vector $\w$. The convex duals $\check\sigma_h$ of the cones $\sigma_h$ are an increasing nested sequence of regular convex cones of dimension $r$ contained in $\check \w$ whose integral points $\check\sigma_h\cap\Z^r$ form nested free semigroups $\N^r_h$ which fill up $\Phi_{\geq 0}$ because the hyperplane $\sum_{i=1}^r\tau_ia_i=0$ has no integral point except the origin.
\end{remark}
\emph{We can now interpret Proposition \ref{OstroAbh} as showing that, if $\nu$ is a rational valuation with finitely generated semigroup and $R$ is complete,  just like in the case of branches {\rm (}see \cite{G-T} and \cite{GP}{\rm )}, the formal space corresponding to $R$ is obtained by deforming the parametrization $u_i\mapsto \rho_it^{\gamma_i}$ of the formal space corresponding to the binomial equations of ${\rm gr}_\nu R$ in $\widehat{k[(u_i)_{i\in I}]}$, which by \cite[Section 3.1]{Te3} is isomorphic to the space corresponding to $k[[t^\Gamma]]$, into $$u_i\mapsto\xi_i (t)=\rho_i t^{\gamma_i}+\sum_{\delta>\gamma_i}c^{(i)}_\delta t^\delta \ {\rm with}\  c^{(i)}_\delta\in k,\ \rho_i\in k^*$$ inside the space with coordinates $(u_i)_{i\in I}$.} \par\smallskip\noindent
Our purpose in this text is to show that the same is true for any rational valuation, at least in the rank one case.
\par\smallskip
 If we only assume that $R\subset R_\nu$ is excellent and analytically irreducible with a field of representatives $k\subset R$ of its algebraically closed residue field (see \cite[Proposition 7.2]{Te2}), the valuation $\nu$ is induced by the embedding $R\subset \hat R^m\subset k[[t^{\Phi_{\geq 0}}]]$ from the $t$-adic valuation $\nu_t$ of the last ring.  \par\medskip\noindent
\end{section}
\begin{section}{Approximating a rational valuation by semivaluations with finitely generated semigroup}\label{App}
In this section, we show how one can apply the valuative Cohen Theorem to construct a sequence of valuations with finitely generated semigroup on quotients of $R$ which provide better and better approximations to the valuation $\nu$. The rank one and algebraically closed residue field hypotheses are not needed.\par\smallskip\noindent
Let us choose a finite subset $B_0$ of $I$ which contains:\begin{itemize}
\item The set $J$ of indices of the elements $(\xi_i)_{i\in J}$ minimally generating the maximal ideal of $R$;
\item A set of $r={\rm rat.rk.}\Phi$ indices $i_1,\ldots ,i_r$ such that the $(\gamma_{i_t})_{t=1,\ldots ,r}$ rationally generate the group $\Phi$;
\item The indices of the finite set of variables $u_i$ appearing in the equations $(F_q)_{1\leq q\leq f}$.
\end{itemize}
\par\medskip

It is stated in \cite[Part 1]{Te4} and proved in \cite{Te3} that there is a nested sequence of finite subsets of $I$:
$$B_0\subset \cdots \subset B_a\subset B_{a+1}\subset\cdots \subset I,$$
such that $\bigcup_{a\in \N}B_a=I$ and which approximate the countable ordinal $I$ in such a way that the following properties hold:\par\noindent We denote by $B$ one of the $B_a$ and by $\iota_{B}$ the injection $k[[(u_i)_{i\in B}]]\subset k[[(u_i)_{i\in I}]]$ and consider in $R$ the ideal $K_B$ of $R$ generated by the images  by the morphism of $k$-algebras $$\pi\circ\iota_B\colon k[[(u_i)_{i\in B}]]\longrightarrow R\ ; \ \ u_i\mapsto \xi_i$$ of the $(F_\ell)_{\ell\in L}$ whose initial binomial involves only variables whose index is in $B$ and in which all the $u_i$ with $i\notin B$ have been set equal to $0$.
 \par\noindent 
 These series are denoted by $F_\ell\vert B$. Note that the $F_\ell$ which use only variables in $B$ are mapped to zero in this operation and the images by $\pi\circ\iota_B$ of the other $F_\ell\vert B$ are contained in the ideal generated by the $(\xi_i)_{i\notin B}$. 
 \par The ideal $K_B$ is the image by $\pi\circ\iota_B$ of the ideal $\Kk_B$ of $k[[(u_i)_{i\in B}]]$ generated by the $F_\ell\vert B$. \par\medskip
  Consider the following commutative diagram. 
\[\xymatrix{&k[[(u_i)_{i\in B}]]\ar[drr]_{\pi_B} \ar @{^{(}->}[r]^{\iota_B}&\widehat{k[(u_i)_{i\in I}]} \ar[r]^\pi&  R\ar[d]^{\kappa_B}\\
              & && R/K_B}\]
  By \cite[Lemma 4.2]{Te3}, the kernel of the surjective morphism $\pi\circ\iota_B$ is generated by the $(F_i)_{i\in B\setminus J}, (F_q)_{1\leq q\leq f}$ and the kernel of $\pi_B$ is generated by the $F_\ell\vert B$ for those $F_\ell$ whose initial binomial is in $k[[(u_i)_{i\in B}]]$, \textit{which include the generators of the kernel of the morphism $\pi\circ\iota_B$.}\par\medskip\noindent
  We now recall the main results in \cite[Section 4]{Te3}. See also \cite[Part 1]{Te4}.
    \begin{theorem}\label{A}{\rm (Approximation of $\nu$ by semivaluations with finitely generated semigroup)}
     \begin{enumerate}
\item The ideals $K_B$ are prime. The quotients $R/K_B$ are of dimension equal to the rational rank $r(\nu)$ of $\nu$ and each is endowed with an Abhyankar valuation $\nu_B$ whose semigroup of values $\Gamma_B$ is finitely generated by the $(\gamma_i)_{i\in B}$. The valuation $\nu_B$ is the image by the morphism $\pi_B$ of the weight valuation of $k[[(u_i)_{i\in B}]]$. 
\item We have the equality $\nu (K_B)={\rm min}_{i\notin B}\gamma_i$, and $\bigcap_{a\in \N}K_{B_a}=(0)$.
\item For any $x\in R\setminus\{0\}$ there is an $a(x)\in \N$ such that for $a\geq a(x)$, the equality $\nu (x)=\nu_{B_a}(x\ {\rm mod.}K_{B_a})$ holds.
\end{enumerate}
  \end{theorem}
\end{section}
\begin{section}{Projective systems of torific uniformizers}\label{PSTI}
In this section we study the relation between torific uniformizers associated to two finite subsets $B_a\subset B_{a+1}\subset I$ of variables as in section \ref{App}. As usual $b_a$ denotes the cardinality of $B_a$.\par\noindent Let $\Ll^{b_a-r}\subset \Z^{b_a}$ be the subgroup generated by the vectors $m^\ell-n^\ell$ of $ \Z^{b_a}$. The corresponding binomials $u^{m^\ell}-\lambda_\ell u^{n^\ell}$ generate the kernel of the morphism $k[(u_i)_{i\in B_a}]\to {\rm gr}_{\nu_{K_{B_a}}}R/K_{B_a}$. Since this last ring is of dimension $r$, the group $\Ll^{b_a-r}$ is of rank $b_a-r$.\par
By \cite[Theorem 2.1]{E-S}, since ${\rm gr}_{\nu_{K_{B_a}}}R/K_{B_a}$ is a domain, the lattice  $\Ll^{b_a-r}$ is saturated in $ \Z^{b_a}$ and we have a split exact sequence of $\Z$-modules
$$(0)\longrightarrow \Ll^{b_a-r}\longrightarrow\Z^{b_a}\longrightarrow\Z^r\to (0).$$
Since by construction $\Z^{b_a}$ is saturated in $\Z^{b_{a+1}}$, it follows that $\Ll^{b_a-r}$ is saturated in $\Ll^{b_{a+1}-r}$ and in $\Z^{b_{a+1}}$.\par\noindent 
The group $\Z^r$ is by construction the group generated by the images of the elements of the natural basis of $\Z^{b_a}$ in the quotient $\Z^{b_a}/\Ll^{b_a-r}$. It is the group generated by the semigroup $\Gamma_{B_a}=\langle(\gamma_i)_{i\in B_a}\rangle$.\par\noindent The diagram 
\[\xymatrix{&\ar[dd] \Z^{b_a}& \longrightarrow \ \ & \ar[dd]^{\rm Identity} \Z^r 
             \\
         & &   &&   \\
             &\Z^{b_{a+1}}&\longrightarrow\ &\ \Z^r } \]
             where the left vertical arrow comes from the natural injection of bases, does not commute.\par\noindent For example the semigroup $\Gamma_0=\langle 2,3\rangle$ generates $\Z$ as a quotient of $\Z^2$ by $\Z(3,-2)$. The semigroup $\Gamma_1=\langle 4,6,13\rangle$ also generates $\Z$ as a quotient of $\Z^3$ by the subgroup generated by $(3,-2,0)$ and $(5,1,-2)$. This corresponds to the fact that the subgroup of $\Phi$ generated by the $(\gamma_i)_{i\in B_a}$ may grow with $a$.\par
 Consider the $r$-dimensional weight cones $W_a\subset \check\R^{b_a}$ and $W_{a+1}\subset \check\R^{b_{a+1}}$. Denoting  by $\pi_{a+1,a}\colon\check\Z^{b_{a+1}}\to\check\Z^{b_a}$ the natural projection as well as its extension $\check\R^{b_{a+1}}\to\check\R^{b_a}$, we see that we have the inclusion $W_{a+1}\subset\pi_{a+1,a}^{-1}(W_a)$.\par\noindent Let us denote by $W_a^{\rm aff}=(\Ll^{b_a-r})^\perp$ the affine hull of $W_a$ in $\check\R^{b_a}$.
\begin{lemma}\label{Latt} \begin{enumerate}
\item The dual $\check\Z^r\subset\check\Z^{b_a}$ of the quotient $\Z^{b_a}/\Ll^{b_a-r}$ is the intersection with $W^{\rm aff}_a$ of the integral lattice $\check\Z^{b_a}$ of $\check\R^{b_a}$.
\item The equality $\Ll^{b_{a+1}-r}\cap \Z^{b_a}=\Ll^{b_a-r}$ holds and entails a natural injection $\Z^{b_a}/\Ll^{b_a-r}\subset\Z^{b_{a+1}}/\Ll^{b_{a+1}-r}$. 
\item The injection $\Z^{b_a}\subset \Z^{b_{a+1}}$ decomposes as $$\iota\colon\Ll^{b_a-r}\oplus\Z^r\subset \Ll^{b_{a+1}-r}\oplus\Z^r,$$ inducing a $\Z$-linear injection $\Z^r\to\Z^r$.
\item The morphism $W^{\rm aff}_{a+1}\cap\check\Z^{b_{a+1}}\to W^{\rm aff}_a\cap\check\Z^{b_a}$ induced by the projection $\pi_{a+1,a}$ is the dual of the injection $\Z^{b_a}/\Ll^{b_a-r}\subset\Z^{b_{a+1}}/\Ll^{b_{a+1}-r}$.
\end{enumerate}
\end{lemma}
\begin{proof} An integral point of $W_a$ is an element of $\check\Z^{b_a}$ vanishing on $\Ll^{b_a-r}$ and thus an element of the $\Z$-dual of $\Z^{b_a}/\Ll^{b_a-r}$. Conversely, $\check\Z^r\subset \check\Z^{b_a}$ is the module of $\Z$-linear forms on $\Z^{b_a}$ vanishing on $\Ll^{b_a-r}$ and they are integral points of $W_a$.\par The second statement follows from the fact that $\Ll^{b_a-r}$ and $\Ll^{b_{a+1}-r}$ correspond respectively to relations between the $(\gamma_i)_{i\in B_a}$ and $(\gamma_i)_{i\in B_{a+1}}$.\par The third statement follows from then second one since $\Ll^{b_{a+1}-r}$ is a direct factor in $\Z^{b_{a+1}}$ and the image of $\Z^r$ in \break$ \Ll^{b_{a+1}-r}\oplus\Z^r$ by the morphism $\iota$ cannot meet $\Ll^{b_{a+1}-r}$.\par
The last statement follows from statement $1)$ applied to $W_a$ and $W_{a+1}$. 
\end{proof}
\begin{definition} Let $(\xi_i)_{i\in I}$ be a system of representatives in $R$ of the generators of ${\rm gr}_\nu R$ according to the valuative Cohen Theorem. Let $R/K_{B_a}$ and $R/K_{B_{a+1}}$ be equipped with the torific embeddings determined by the images of the $(\xi_i)_{i\in B_a}$ and $(\xi_i)_{i\in B_{a+1}}$ respectively. A torific uniformizer $\Sigma_{a+1}$ for $R/K_{B_{a+1}}$ is \textit{compatible} with a torific uniformizer $\Sigma_a$ for $R/K_{B_a}$ if it is compatible with the $(\pi^{-1}_{a+1,a}(\tau))_{\tau\in\Sigma_a}$. In view of Definition \ref{mizer}, the fan $\Sigma_{a+1}$ is also compatible with $W_{a+1}$. 
\end{definition}
\begin{lemma}\label{compat} Given a torific uniformizer for $R/K_{B_a}$, there exist compatible torific uniformizers for  $R/K_{B_{a+1}}$.
\end{lemma}
\begin{proof} This follows immediately from Proposition \ref{cones}.
\end{proof}
Let us denote by $\w_a$ the vector $(w(u_1),\ldots ,w(u_{b_a}))\in\check\R^{b_a}$ and by $\w_{a+1}$ the vector $(w(u_1),\ldots ,w(u_{b_a}),\ldots, w(u_{b_{a+1}}))\in\check\R^{b_{a+1}}$.\par\noindent By Lemma \ref{weighin}, the fan $\Sigma_a$ must contain a unique $r$-dimensional regular cone $\sigma_{\w_a}$ containing $\w_a$ in its relative interior, and similarly for $\Sigma_{a+1}$ and $\w_{a+1}$. 
Since the fan $\Sigma_{a+1}$ is a uniformizer, it must be compatible not only with the hyperplanes $H_\ell$ but also with the rational convex cones derived from the exponents appearing in the generators of the kernel of the map $k[[(u_i)_{i\in B_{a+1}}]]\to R/K_{B_{a+1}}$ as in Proposition \ref{Resol}.\par\noindent In particular, if $\gamma_{i(a)}={\rm min}_{i\notin B_a}\gamma_i$, this kernel contains the series $F_{i(a)}$ so that the cone $\sigma_{\w_{a+1}}$ must be entirely on the positive side of the hyperplane of $\check\Z^{b_{a+1}}$ linearly dual to the vector $\epsilon_{i(a)}-n^{i(a)}$, where $\epsilon_{i(a)}$ is the basis vector of $\Z^{b_{a+1}}$ corresponding to the variable $u_{i(a)}$.\par
Let us write $\sigma_{\w_{a+1}}=\langle \breve a^1,\ldots,\breve a^r\rangle$.
Since $\Sigma_{a+1}$ is compatible with $\Sigma_a$, we can choose a maximal regular cone $\sigma_{a+1}$ of $\Sigma_{a+1}$ containing $\sigma_{\w_{a+1}}$ and whose image by $\pi_{a+1,a}$ is contained in a maximal regular cone $\sigma_a$ of $\Sigma_a$.
\begin{lemma}\label{inclusions} We must then have $$\sigma_{\w_a}=\sigma_a\cap W_a,\ \sigma_{\w_{a+1}}=\sigma_{a+1}\cap W_{a+1},\  {\rm and}\ \pi_{a+1,a}(\sigma_{\w_{a+1}})\subseteq \sigma_{\w_a},$$
Moreover, $ \pi_{a+1,a}(\sigma_{\w_{a+1}})$ is of dimension $r$.
\end{lemma}\begin{proof} The proof is the same as that of Lemma \ref{weighin} since by construction $\pi_{a+1,a}(\sigma_{\w_{a+1}})$ contains $\w_a$ and is contained in a cone of $\Sigma_a$.
\end{proof}
The following diagram of toric morphisms and inclusions of orbits and points
\[\xymatrix{&\ar[dd]^{\pi_{a+1,a}^*} Z(\Sigma_{a+1})& \hookleftarrow \ \ \ \ar[dd] Z(\sigma_{a+1}) &\hookleftarrow \ \ \ \ar[dd] o(\sigma_{\w_{a+1}}) \owns & \ar[dd] c_{a+1}
             \\
         & &   &&   \\
             & Z(\Sigma_a)&\hookleftarrow\ \  Z(\sigma_a) & \hookleftarrow \ \ \  o(\sigma_{\w_a})\  \owns & c_a} \]
 where the inclusions on the left are affine charts and those on the right are the inclusions of $r$-codimensional orbits and of the centers of $\nu_{B_a}$ in $Z(\sigma_a)$ and $\nu_{B_{a+1}}$ in $ Z(\sigma_{a+1})$, is commutative.\par\noindent
Thus, there are bases $$a^1,\ldots, a^r,a^{r+1},\ldots ,a^{b_a} \ {\rm and}\ \breve a^1,\ldots, \breve a^r,\breve a^{r+1},\ldots , \breve a^{b_{a+1}}$$ of the lattices $\check \Z^{b_a}$ and $\check\Z^{b_{a+1}}$, respectively generating the cones $\sigma_a$ and $\sigma_{a+1}$, and such that the vectors $a^1,\ldots, a^r$ generate the cone $\sigma_{\w_a}$ and $\breve a^1,\ldots, \breve a^r$ generate the cone  $\sigma_{\w_{a+1}}$.\par\noindent
  After Lemma \ref{inclusions}, in these bases the map of lattices $\pi_{a+1,a}\colon \check\Z^{b_{a+1}}\to \check \Z^{b_a}$ is described by:
     $$\begin{array}{lr}
\ \ \ \ \ \pi_{a+1,a}(\breve a^1)\ \ \ =e^1_1a^1+\cdots +e^r_1a^r\\
\ \ \ \ \ \ \ \ .\\
\ \ \ \ \ \ \ \ .\\
\ \ \ \ \ \ \ \ .\\
\ \ \ \ \  \pi_{a+1,a}(\breve a^r)\ \ \ =e^1_ra^1+\cdots +e^r_ra^r\\
\ \ \ \ \ \pi_{a+1,a}( \breve a^{r+1})=e^1_{r+1} a^1+\cdots +e^r_{r+1}a^r+e^{r+1}_{r+1}a^{r+1}+\cdots +e^{b_a}_{r+1}a^{b_a}\\
\ \ \ \ \ \ \ \ .\\
\ \ \ \ \ \ \ \ .\\
\ \ \ \ \ \ \ \ .\\
\ \ \ \ \ \pi_{a+1,a}( \breve a^{b_{a+1}})=e^1_{b_{a+1}} a^1+\cdots +e^r_{b_{a+1}} a^r+e^{r+1}{b_{a+1}} a^{r+1}+\cdots +e^{b_a}_{b_{a+1}} a^{b_a}\\

\end{array}$$

 with $e^j_i\in \Z_{\geq 0}$ since the image of $\sigma_{a+1}$ is contained in $\sigma_a$.\par 
 
 Taking now the dual bases of $\Z^{b_a}$ and $\Z^{b_{a+1}}$, which are respectively in bijection with generators $y_1,\ldots y_{b_a}$ and $\breve y_1,\ldots \breve y_{b_{a+1}}$ of the polynomial algebras corresponding to the free semigroups $\check\sigma_a\cap\Z^{b_a}$ and $\check\sigma_{a+1}\cap \Z^{b_{a+1}}$ and applying the general theory of toric morphisms, we obtain that:\par\noindent
 there are coordinates $y_1, \ldots, y_r,y_{r+1},\ldots ,y_{b_a}$ and $ \breve y_1, \ldots, \breve y_r,\breve y_{r+1},\ldots ,\breve y_{b_{a+1}}$ on $Z(\sigma_a)$ and $Z(\sigma_{a+1})$ respectively, such that the toric morphism\newline $Z(\sigma_{a+1})\to Z(\sigma_a)$ is described by:
$$\begin{array}{lr}
 \ \ \ \ \ \ \ \ \ \ \ \ \ \ \ \ y_1\mapsto\ \ \ \breve y_1^{e^1_1}\ldots \breve y_r^{ e^r_1}
\ \ \ \ \ \ \ \ \ \ \ \ \ \ \ \ .\\
\ \ \ \ \ \ \ \ \ \ \ \ \ \ \ \ .\\
(**)\ \ \ \ \ \ \ \ \ \ \ .\\
\ \ \ \ \ \ \ \ \ \ \ \ \ \ \ \ y_r\mapsto\ \ \ \breve y_1^{e^1_r}\ldots\breve y_r^{e^r_r}\\
\ \ \ \ \ \ \ \ \ \ \ \ \ \ \ \ y_{r+1}\mapsto\breve y_1^{e^1_{r+1}}\ldots \breve y_r^{e^r_{r+1}}\breve y_{r+1}^{e^{r+1}_{r+1}}\ldots \breve y_{b_{a+1}}^{e^{b_{a+1}}_{r+1}}\\
\ \ \ \  \ \ \ \ \ \ \ \ \ \ \ \ .\\
\ \ \ \ \ \ \ \ \ \ \ \ \ \ \ \ .\\
\ \ \ \ \ \ \ \ \ \ \ \ \ \ \ \ .\\
\ \ \ \ \ \ \ \ \ \ \ \ \ \ \ \ y_{b_a}\mapsto\breve y_1^{e^1_{b_a}}\ldots \breve y_r^{e^r_{b_a}}\breve y_{r+1}^{e^{r+1}_{b_a}}\ldots  \breve y_{b_{a+1}}^{e^{b_{a+1}}_{b_a}},\\

\end{array}$$
\noindent
where, as we saw, the $e^j_i$ are the coordinates in the basis of the generators of $\sigma_a$ of the images by $\pi_{a+1,a}$ of the generators, properly numbered,  $\breve a^1,\ldots \breve a^r ,\ldots ,\breve a^{b_{a+1}}$ of the cone $\sigma_{a+1}$.\par\noindent
We note that the matrix of the $(e^j_i), 1\leq i,j\leq r$ is not unimodular in general because $\pi_{a+1,a}(W_{a+1}^{\rm aff}\cap\check \Z^{b_{a+1}})$ may be a strict sublattice of $W_a^{\rm aff}\cap\check \Z^{b_a}$, and thus may correspond to a toric morphism $Z(\sigma_{\w_{a+1}})\to Z(\sigma_{\w_a})$ of degree $>1$.\par
 For $i\in B_a$, consider, in the completion $k[[y_1,\ldots ,y_r]]$ of the algebraic transform $\tilde R_a$ of $R/K_{B_a}$ by the toric modification of $\A^{b_a}(k)$ corresponding to the fan $\Sigma_a$, the image $\xi_i^{(a)}$ of $\xi_i$ by the composed morphism
 $$R\longrightarrow R/K_{B_a}\hookrightarrow \tilde R_a\hookrightarrow k[[y_1,\ldots ,y_r]].$$ By Proposition \ref{OstroAbh}, we have
$$\xi_i^{(a)}=y_1^{a^1_i}\ldots y_r^{a^r_i}(c_{r+1}+z_{r+1})^{a^{r+1}_i}\ldots (c_{b_a}+z_{b_a})^{a^{b_a}_i},\leqno{(S)}$$
where $c_j\in k^*$ and the $z_j$ are power series without constant term in $y_1,\ldots ,y_r$. The $c_j$ depend on $a$ since they are coordinates of the center of the valuation in $o(\sigma_{\w_a})$.\par
If we embed $k[[y_1,\ldots ,y_r]]$ in $k[[t^{\Phi_{\geq 0}}]]$ by $y_i\mapsto t^{\nu(y_i)}$, we have ${\rm in}_{\nu_t} \xi_i^{(a)}=c_{r+1}^{a^{r+1}_i}\ldots c_{b_a}^{a^{b_a}_i}t^{\gamma_i}$, where $\nu_t$ denotes the $t$-adic valuation of $k[[t^{\Phi_{\geq 0}}]]$.
\end{section}

\begin{section}{The embedded Kaplansky embedding Theorem for rational valuations of rank one}\label{final}
In this section we show that if the rational valuation $\nu$ is of rank one, the embedded Kaplansky embeddings for the Abhyankar semivaluations on the quotients $R/K_{B_a}$ give rise to embedded Kaplansky embeddings for $R$. Our index set $I$ is assumed to be infinite and of ordinal $\omega$.\par In characteristic zero, examples of rank one valuations with infinitely generated semigroup on the ring $R=k[[x,y]]$ can be obtained from series $y(x)\in k[[x^{\Q_{\geq 0}}]]$ with a well ordered set of exponents having unbounded denominators. By Newton-Puiseux, for any non-zero $h(x,y)\in k[[x,y]]$ we have $h(x,y(x))\neq 0$ and its order in $x$ determines the valuation of $h(x,y)$, with values in $\Q_{\geq 0}$. Generators of the ideals $K_B\subset k[[x,y]]$ of section \ref{App} then appear as the equations of the formal branches whose Puiseux expansions are suitable truncations, with exponents having bounded denominators, of the series $y(x)$. The truncations must correspond to keeping only the exponents which have a common denominator. The generators $\gamma_i$ of the semigroup are the valuations of these equations, which corresponds to the first part of Assertion $2)$ of Theorem \ref{A}. For references and more details see \cite[Example 4.19]{Te1}.\par  A different approach to the Kaplansky embedding theorem for rings in the mixed characteristic case can be found in the work of San Saturnino, see \cite[Theorem 4.1]{S}. We believe that the method presented here, beginning with the valuative Cohen Theorem, can be extended to the mixed characteristic case.
\begin{theorem}\label{Kap}{\rm (Embedded Kaplansky embedding for rational valuations of rank one)}
Let $\nu$ be a zero dimensional valuation of rank one on a complete noetherian equicharacteristic local domain $R$ with algebraically closed residue field $k$. Fix a field of representatives $k\subset R$. Let $(\xi_i)_{i\in I}$ in $R$ be such that the valuative Cohen Theorem holds and let $$\pi\colon\widehat{k[(u_i)_{i\in I}]}\to R, \ u_i\mapsto\xi_i$$
be the presentation given by the valuative Cohen Theorem. Let $\Phi$ denote the value group of the valuation. Then there exist Hahn series $\xi_i(t)\in k[[t^{\Phi_{\geq 0}}]]$, $$\xi_i(t)=\rho_i t^{\gamma_i}+\sum_{\delta>\gamma_i}c^{(i)}_\delta t^\delta \ {\rm with}\  c^{(i)}_\delta\in k,\ \rho_i\in k^*$$ such that $R$ is the image of the morphism $$\widehat{k[(u_i)_{i\in I}]}\to k[[t^{\Phi_{\geq 0}}]],\  u_i\mapsto \xi_i \mapsto\xi_i(t).$$
The valuation $\nu$ on $R$ is induced from the $t$-adic valuation of $ k[[t^{\Phi_{\geq 0}}]]$ through the inclusion $R\subset k[[t^{\Phi_{\geq 0}}]]$ determined by $\xi_i\mapsto\xi_i(t)$.
\end{theorem}
\begin{proof} 
We juxtapose the formulas $(S)$ at the end of the preceding section for $B_a$ and $B_{a+1}$ and map $R$ to $k[[t^{\Phi_{\geq 0}}]]$ through  $R/K_{B_{a+1}}$ by mapping $\breve y_i$ to $t^{\nu (\breve y_i)}$. Then we map $\tilde R_a$ to $\tilde R_{a+1}$ and so $k[[y_1,\ldots ,y_r]]$ to $k[[\breve y_1,\ldots ,\breve y_r]]$ through the morphism described by $(**)$. \par\noindent The images $\xi_i^{(a)}$ and $\xi_i^{(a+1)}$ now both have expressions as series in the $\breve y_i$ with $1\leq i\leq r$. \par\noindent We define embeddings of $k$-algebras $k[[y_1,\ldots ,y_r]]\hookrightarrow k[[t^{\Phi_{\geq 0}}]]$ and $k[[\breve y_1,\ldots ,\breve y_r]]\hookrightarrow k[[t^{\Phi_{\geq 0}}]]$ by $y_s\mapsto t^{\nu (y_s)}$ and $\breve y_s\mapsto t^{\nu (\breve y_s)}$ respectively. \par\smallskip\noindent
In view of the monomial nature of the morphism described by $(**)$, the first embedding coincides with the composition 
$$k[[y_1,\ldots ,y_r]]\hookrightarrow k[[\breve y_1,\ldots ,\breve y_r]]\hookrightarrow k[[t^{\Phi_{\geq 0}}]],\ \ \breve y_s\mapsto t^{\nu (\breve y_s)}.$$
We may now view the images of $\xi_i^{(a)}$ and $\xi_i^{(a+1)}$ as series in $t$, denoted by $\xi_i^{(a)}(t), \xi_i^{(a+1)}(t)$.\par
   It follows from the definition of the ideals $K_{B_a}$ that we have the inclusions $K_{B_a}\subset K_{B_{a+1}}+((\xi_j)_{j\notin B_a})$: indeed, the difference between the $F_\ell\vert B_{a+1}$ and the $F_\ell\vert B_a$ concerns only those $F_\ell$ whose initial binomial either\begin{enumerate} \item is in $k[(U_i)_{i\in B_a}]$ and in those, in the definition of the $F_\ell\vert B_{a+1}$, we keep the variables $u_i$ with $i\in B_{a+1}\setminus B_a$ so that the images in $R$ differ by an element of the ideal generated by the $(\xi_j)_{j\in B_{a+1}\setminus B_a}$, or
\item  involves a variable $u_i$ with $i\in B_{a+1}\setminus B_a$ and then its image in $R$ is in the ideal generated by the $(\xi_j)_{j\notin B_{a+1}}$. 
\end{enumerate}
In both cases, their images in $R$ therefore differ only by an element of the ideal generated by $(\xi_j)_{j\notin B_a}$. The valuation of this ideal is ${\rm min}_{j\notin B_a}\gamma_j$.\par\smallskip

By Theorem \ref{A}, for each index $i\in I$ there is a smallest integer $a(i)$ depending on $i$ such that $i\in B_{a(i)}$, $\xi_i\notin K_{B_{a(i)}}$ and $\nu_{B_a}(\xi_i \ {\rm mod.} K_{B_a})=\gamma_i$ for $a\geq a(i)$.\par\noindent Given $a\geq a(i)$, the images of $\xi_i$ in $R/K_{B_a}$ and $R/K_{B_{a+1}}$ have the same image in  $R/K_{B_{a+1}}+((\xi_j)_{j\notin B_a})$. The series $\xi_i^{(a)}(t)$ and $\xi_i^{(a+1)}(t)$ with $i\in B_a$ which we have just seen are therefore equal modulo the ideal of $k[[t^{\Phi_{\geq 0}}]]$ generated by the  $(\xi_j^{(a+1)}(t))_{j\notin B_a}$  and thus differ by an element of $t$-adic value $\geq {\rm min}_{j\notin B_a}\gamma_j$. Let us denote this value by $\gamma_a$.\par\noindent The $\gamma_i$ are cofinal in $\Gamma$ because $\Gamma$ has ordinal $\omega$ and for every $a$ there is a $t(a)\in\N$ such that  have $\gamma_i<\gamma_a<\gamma_{a+t(a)}$.\par\noindent From the first inequality we deduce that ${\rm in}_{\nu_t}\xi_i^{(a+1)}(t)= {\rm in}_{\nu_t}\xi_i^{(a)}(t)$ so that ${\rm in}_\nu \xi_i$ has a well defined image in ${\rm gr}_{\nu_t}k[[t^{\Phi_{\geq 0}}]]=k[t^{\Phi_{\geq 0}}]$. This entails relations between the $c_j$ and the $\breve c_j$, which follow from  the fact that they are coordinates of the centers of the same valuation. \par\noindent
With our usual notations, the preceding discussion implies
$$\xi_i^{(a+1)}(t)-\xi_i^{(a)}(t)\in \Pp_{\gamma_a}( k[[t^{\Phi_{\geq 0}}]]).\leqno{(PC)}$$

Given $i\in I$, we start from $B_{a(i)}$. Using Lemma \ref{compat} we build a projective system of torific embeddings as in section \ref{PSTI}. We obtain a sequence of series $(\xi_i^{(a)}(t))_{a\geq a(i)}$ satisfying the inclusions $(PC)$.\par\noindent
It follows that, perhaps after restricting the indices $a$ to a cofinal subset of $\N$, the $(\xi_i^{(a)}(t))_{a\geq a(i)}$ form a pseudo-convergent sequence in $k[[t^{\Phi_{\geq 0}}]]$.\par\noindent By \cite[Section 3]{C-T}, the semigroup $\Gamma$ can have no accumulation point in $\R$, so that the breadth of the pseudo-convergent sequence is zero and its limit is an element $\xi_i(t) \in k[[t^{\Phi_{\geq 0}}]]$ which is well defined and independent of $a\geq a(i)$. Its $\nu_t$-initial form can be denoted by $\rho_i t^{\gamma_i}$. \par
 As we saw in section \ref{secCoh}, the map of graded $k$-algebras determined by ${\rm in}_\nu \xi_i\mapsto \rho_i t^{\gamma_i}$ establishes an isomorphism between ${\rm gr}_\nu R$ and $k[t^\Gamma]\subset k[t^{\Phi_{\geq 0}}]$. This implies that the morphism of $k$-algebras $R\to k[[t^{\Phi_{\geq 0}}]]$ determined by $\xi_i\mapsto\xi_i(t)$ is injective.\end{proof}
 \begin{subsection} {Remarks}

1) As the index $a\geq a(i)$ grows, the images in $\sigma_{\w_{a(i)}}$ of the $\sigma_{\w_{a+1}}$ by the compositions of the maps $\pi_{a+1,a}$ become smaller and smaller and the value semigroups $\N^r\subset \Phi_{\geq 0}$ of the $k[[y_1,\ldots ,y_r]]$ become larger and larger. If for some $a_0$ the image of the cone $\sigma_{\w_{a+1}}$ is equal to $\sigma_{\w_a}$ for $a\geq a_0$, the semigroup $\Phi_{\geq 0}$ is equal to $\N^r$ so that it is well ordered and this can happen only if $r=1$.\par\smallskip\noindent
2) The same argument shows that if we had started from a different sequence of finite subsets $B'_{a'}$ approximating the ordinal $I$, and thus a different sequence of ideals $K_{B'_{a'}}$, the limits $\xi_j(t)$ would be the same since we have inclusions of the form $B_a\subset B'_{a'}\subset B_{a+k}\subset B'_{a'+\ell}$.\par\smallskip\noindent 
3) There remains the problem of understanding to what extent the embedded Kaplansky embeddings are independent, up to an inner automorphisme of $k[[t^{\Phi_{\geq 0}}]]$, of the projective system of torific uniformizers.\par\smallskip\noindent
4) The valued fields which are the fraction field of an excellent, and in particular noetherian, local domain with an algebraically closed residue field are rather special. Apart from the completeness, the noetherianity of $R$ plays an essential role in this text, first by the valuative Cohen Theorem, and second by the fact that the semigroups of values are well ordered and have no accumulation points in $\Phi\otimes_\Z\R$. It would be interesting to revisit from this perspective, beginning with \cite{Ka} and \cite{Ka2}, the existing counterexamples to the uniqueness of Kaplansky embeddings in the theory of valued fields. This applies also to the condition $\Phi=p\Phi$ when $k$ is of characteristic $p$ as suggested by examination of the proofs of Lemmas 12 and 13 in \cite{Ka} or of example \ref{A_S} when one replaces $k((x))$ by its perfect closure.\par\smallskip\noindent
5) The author is convinced that there is a result similar to Theorem \ref{Kap} for rank $>1$ but then the breadths of the pseudo-convergent sequences built from the $R/K_{B_a}$ are more difficult to understand and therefore so is what determines the embedded Kaplansky embedding.
\end{subsection}

\subsection{Aknowledgements} The author is very grateful to the referees, especially the one known to him as $X_0$, whose informed and constructive criticisms have led to very substantial improvements in the text.

\end{section}

\end{document}